\documentclass[12pt]{article}

\RequirePackage{fancymath}

\usepackage[affil-it]{authblk}

\makeatletter
\def\@maketitle{%
  \newpage
  \null
  \begin{center}%
  \let \footnote \thanks
    {\Large\bfseries \@title \par}%
    \vskip 1.5em%
    {\normalsize
      \lineskip .5em%
      \begin{tabular}[t]{c}%
        \@author
      \end{tabular}\par}%
    \vskip 0.8em%
    {\small \@date}%
  \end{center}%
  \par
  \vskip 1.5em}
\makeatother

\title{The Riemann integral on Dedekind complete $f$-algebras}

\author{Eder Kikianty%
\thanks{Email: \texttt{eder.kikianty@wits.ac.za}}}
\affil[1]{School of Mathematics, University of the Witwatersrand, Private Bag 3, WITS 2050, South Africa}

\author[2]{Luan Naude%
\thanks{Email: \texttt{luan.naude98@gmail.com}}}
\affil[2]{Department of Mathematics and Applied Mathematics, University of Pretoria, Private Bag X20, Hatfield 0028, South Africa}

\author{Mark Roelands%
\thanks{Email: \texttt{m.roelands@math.leidenuniv.nl}}}
\affil[3]{Mathematical Institute, Leiden University, 2300 RA Leiden,
The Netherlands}

\author[2]{Christopher Schwanke%
\thanks{Email: \texttt{cmschwanke26@gmail.com}}}

\date{\today}

\begin{document}

\maketitle

\vspace{-.8cm}

\begin{abstract}
{\scriptsize In this paper we develop a theory of integration for locally band preserving functions, introduced by Ercan and Wickstead, on Dedekind complete $f$-algebras. Specifically, we construct Darboux and Riemann integrals and show that they are equal. We then connect the theory of integrable functions to the theory of order differentiable functions, introduced by the third and fourth authors, by proving a Fundamental Theorem of Calculus. Furthermore, we show that a Mean Value Theorem for Integrals holds and that we can integrate by parts and substitutions.}  
\end{abstract}

{\footnotesize {\bf Keywords:} Riemann integral, vector lattice, $\Phi$-algebra, differentiation, locally band preserving, Fundamental Theorem of Calculus} 

{\footnotesize {\bf Subject Classification:} Primary: 46A40; Secondary:  46G12}

\section{Introduction}

Recently, an order-theoretic approach to complex analysis was developed in \cite{series, differentiation} for complex $\Phi$-algebras (unital $f$-algebras). The first paper focused on series and power series in universally complete complex $\Phi$-algebras, and the second paper introduced two distinct, but both natural, notions of differentiability using order convergence in Dedekind complete $\Phi$-algebras: \emph{order differentiability} and \emph{super order differentiability}. This theory allows one to do analysis on certain spaces which do not have a natural norm, such as $C^\infty(K)$, and it illustrates how order-theoretic notions play a crucial role in the development of the theory of complex analysis.

A natural continuation of the work done in \cite{series, differentiation} is to develop a theory of differential and integral calculus in real Dedekind complete $\Phi$-algebras. The work done on order derivatives in \cite{differentiation} naturally carries over to the real case, and this paper mainly focuses on the construction and the properties of the Riemann integral.

While developing the theory of Riemann integration, \emph{locally band preserving functions}, first introduced in \cite{Wickstead-Zafer}, became central to the theory. We further developed their properties in \cite{paper1}. Generalisations of classical results were proved for locally band preserving functions in real Dedekind complete $\Phi$-algebras, such as the Intermediate Value Theorem, the Extreme Value Theorem, and the Mean Value Theorem---the last of which uses the notion of order differentiability introduced in \cite{differentiation}.

The structure of the paper is as follows. Some preliminary results are given in Section 2. Primarily, these are results from \cite{paper1} that are needed in the remainder of the paper.

The Riemann integral for locally band preserving and order bounded functions $f \colon [a, b] \to E$ is constructed in Section 3, where $E$ is a Dedekind complete $f$-algebra and $[a, b]$ is an order closed interval in $E$. As in the classical case, we show that the Darboux and Riemann constructions yield the same integral.

In Section 4, we prove analogues of the basic classical properties of the integral. We also show that a locally band preserving and uniformly order continuous function is necessarily integrable. It is known that an order continuous function need not be uniformly order continuous in this setting, so it is an open question whether order continuity is sufficient to guarantee integrability.

Section 5 is devoted to defining $\int_a^b f(x) dx$ when $a$ and $b$ are incomparable. To do this, we use the bands $B_{a < b}$, $B_{b < a}$, and $B_{a = b}$, which were introduced in \cite{paper1}, to mimic the law of trichotomy in $\bR$ where we must either have $a < b$, $b < a$, or $a = b$.

Finally, we establish a connection between the integral and the differentiation theory from \cite{differentiation} by proving a Fundamental Theorem of Calculus in Section 6. We also prove a Mean Value Theorem for Integrals, and use the fundamental theorem to generalise classical integration techniques, namely substitution and integration by parts.

\section{Preliminaries}






For any unexplained terminology or basic results in vector lattices and $f$-algebras, the reader is referred to the standard texts \cite{aliprantis,zaanen1,depagter,zaanen2}.

\begin{notation}
$E$ denotes a real Archimedean vector lattice. We consider only functions mapping a subset of $E$ into $E$, so the notation $f \colon \dom(f) \to E$ implies $\dom(f) \subseteq E$. Any notation used to denote a band will carry over to its band projection, replacing the $B$ with a $\bP$, e.g. the band projections of bands $B$, $B_1$, and $B_x$ will be written as $\bP$, $\bP_1$, and $\bP_x$, respectively.
\end{notation}

\subsection{Band decompositions of $E$}

In \cite[Notation 2.4]{paper1}, the notation $B_{x \leq y}$, $B_{x = y}$, and $B_{x < y}$ for $x,y \in E$ was introduced. By \cite[Proposition 2.7]{paper1}, these can be defined as follows.

\begin{definition}
Let $E$ be a vector lattice with the projection property, and consider $x, y \in E$. We write
\begin{enumerate}[(i)]
\item $B_{x \leq y}$ for the largest band $B$ for which $\bP(x) \leq \bP(y)$,
\item $B_{x = y}$ for the largest band $B$ for which $\bP(x) = \bP(y)$, and
\item $B_{x < y}$ for the largest band $B$ for which $\bP(y - x)$ is a weak order unit in $B$.
\end{enumerate}
\end{definition}

It is shown in \cite[Proposition 2.6]{paper1} that $E = B_{x < y} \oplus B_{y \leq x} = B_{x < y} \oplus B_{y < x} \oplus B_{x = y}$.

In Section 3, we will work with partitions of an order interval $[a, b]$. Classically, the union of two partitions is again a partition, but this is not the case in $E$ since the union may fail to be totally ordered. For this reason, the following lemma and the total orderisation of a finite set will be useful.

\begin{lemma} \label{l:totally_ordered_decomposition}
Let $E$ be a vector lattice with the projection property, and let $A = \{x_1, \dots, x_n\} \subseteq E$. Then there exists a decomposition $E = \bigoplus_{i = 1}^m B_i$ such that $\bP_i(A)$ is totally ordered for $1 \leq i \leq m$.
\end{lemma}
\begin{proof}
The collection of bands $B$ of the form
\[ B \defeq \bigcap_{ 1 \leq i, j \leq n} B_{i, j} \]
where $B_{i, j} \in \{ B_{x_i < x_j}, B_{x_i \geq x_j} \}$ is such a decomposition of $E$. By definition these bands are disjoint, and that they sum to $E$ follows from the fact that $B_{x_i < x_j} \oplus B_{x_i \geq x_j} = E$ and the distributivity of sums and intersections of bands. Consider the band projection $\bP$ corresponding to such a $B$. Then, for $1 \leq k, l \leq n$, we have that $\bP \leq \bP_{k, l}$ and either $\bP_{k, l}(x_k) \leq \bP_{k, l} (x_l)$ or $\bP_{k, l}(x_l) \leq \bP_{k, l} (x_k)$. Hence, $\bP(A)$ is totally ordered.
\end{proof}

The total orderisation $\totord(A)$ of a nonempty finite subset $A$ of a distributive lattice was defined in \cite[Definition 2.4]{totord}.

\begin{definition}
Let $L$ be a distributive lattice, with $x_1, \dots, x_n \in L$. For $1 \leq k \leq n$, define
\[ \cM_k(x_1, \dots, x_n) := \bigvee_{1 \leq i_1 < \cdots < i_{n+1-k} \leq n} \left( \bigwedge_{j = 1}^{n+1-k} x_{i_j} \right). \]

The set $\{\cM_1(x_1, \dots, x_n), \dots, \cM_n(x_1, \dots, x_n)\}$ is called the \emph{total orderisation} of $\{x_1, \dots, x_n\}$ and denoted $\totord(\{x_1, \dots, x_n\})$.
\end{definition}

For a totally ordered set $\{x_1, \dots, x_n\}$, $\cM_k(x_1, \dots, x_n)$ would simply be the $k$-th smallest element of the set. Therefore, in light of \Cref{l:totally_ordered_decomposition}, an alternative construction of the total orderisation exists for vector lattices with the projection property---where we select the $k$-th smallest element in each of the bands $B_i$ and add them together to get $\cM_k(x_1, \dots, x_n)$. From this construction, the following lemma follows easily.



\begin{lemma} \label{l:totord_decomposition}
Let $E$ be a vector lattice with the projection property, and let $A = \{x_1, \dots, x_n\} \subseteq E$. Then $\totord(A)$ is the unique totally ordered subset $A' \subseteq E$ such that there exists a decomposition $E = \bigoplus_{i = 1}^m B_i$ such that $\bP_i(A) = \bP_i(A')$ for $1 \leq i \leq m$.
\end{lemma}

\subsection{Order intervals}

We introduce the notation we use for order intervals in $E$. It is natural for order closed intervals to use $\leq$, but for order open intervals it is more natural to use a relation we denote $\ll$ rather than $<$. This was done already in \cite{paper1, series, differentiation}.

\begin{notation}
Let $a, b \in E$. The notation $a \ll b$ means that $b - a$ is a weak order unit of $E$. If $a \leq b$, we write
\[ [a, b] := \{x \in E: a \leq x \leq b \}. \]
If $a \ll b$, we write
\[ (a, b) := \{x \in E: a \ll x \ll b \}. \]
Whenever we write $[a, b]$ (respectively $(a, b)$)---for example, as the domain of a function---the reader may assume that $a, b \in E$ and $a \leq b$ (respectively $a \ll b$).
\end{notation}

\subsection{Order continuity and uniform order continuity}

In \cite[Proposition 2.13]{paper1}, order continuity was characterised in a way that resembles the classical $\eps-\delta$ definition of continuity. A similar argument can be used for uniformly order continuous functions. We give these characterisations below in \Cref{p:eps_delta_continuity}.

\begin{definition}
Consider $f \colon \dom(f) \to E$. We say $f$ is \emph{order continuous at $c \in \dom(f)$} if whenever $x_\alpha \to c$ in $\dom(f)$, we have $f(x_\alpha) \to f(c)$. The function $f$ is called \emph{order continuous} if it is order continuous at every point in its domain. We say $f$ is \emph{uniformly order continuous} if whenever $x_\alpha - y_\alpha \to 0$ in $\dom(f)$, we have $f(x_\alpha) - f(y_\alpha) \to 0$.
\end{definition}

\begin{notation}
The notation $\Delta \downarrow 0$ means that $\Delta$ is a downwards directed set in $E$ with infimum $0$, whereas $\Eps \searrow 0$ means that $\Eps$ is a subset of $E$ with infimum $0$ that needs not be directed.
\end{notation}

In the following characterisations of order continuity and uniform order continuity, and later in \Cref{d:order_differentiable,d:super_order_differentiabile}, we use $\Delta \downarrow 0$ and $\Eps \searrow 0$. It is possible to replace $\Eps \searrow 0$ with $\Eps \downarrow 0$, but $\Delta$ must be downwards directed. Therefore, the different arrows indicate when directedness is required

\begin{proposition} \label{p:eps_delta_continuity}
Let $E$ be a real vector lattice. A function $f \colon \dom(f) \to E$ is
\begin{enumerate}[(i)]
\item order continuous at $c \in \dom(f)$ if and only if for every $\Delta \downarrow 0$ there exists an $\Eps \searrow 0$ with the property that for all $\eps \in \Eps$ there exists a $\delta \in \Delta$ satisfying
\[ x \in \dom(f) \text{ and } \abs{x - c} \leq \delta \implies \abs{f(x) - f(c)} \leq \eps. \]
\item uniformly order continuous if and only if for every $\Delta \downarrow 0$ there exists an $\Eps \searrow 0$ with the property that for all $\eps \in \Eps$ there exists a $\delta \in \Delta$ satisfying
\[ x, y \in \dom(f) \text{ and } \abs{x - y} \leq \delta \implies \abs{f(x) - f(y)} \leq \eps. \]
\end{enumerate}
\end{proposition}

\subsection{Locally band preserving functions}

The notion of a locally band preserving function was first introduced in \cite{Wickstead-Zafer}, and was investigated further in \cite{paper1}. In Section 3, we will construct a Riemann integral for order bounded functions $f \colon [a, b] \to E$, and we will investigate the integrability of $f$ primarily when $f$ is locally band preserving. What follows is some of the properties of locally band preserving functions that will be needed throughout the remainder of the paper. We use a definition that is equivalent to the one used in \cite{Wickstead-Zafer} by \cite[Proposition 4.4.]{paper1} for vector lattices with the projection property.

\begin{definition}
Let $E$ be a vector lattice with the projection property. A function $f \colon \dom(f) \to E$ is \emph{locally band preserving} if for every band projection $\bP$ and for all $x, y \in \dom(f)$, we have that $\bP(x) = \bP(y)$ implies $\bP(f(x)) = \bP(f(y))$.
\end{definition}

\begin{example}\phantom{x} 
\begin{enumerate}[(i)]
\item A linear function $f \colon E \to E$ is locally band preserving if and only if it is band preserving.
\item Polynomials and power series, as in \cite{series}, defined on $E$ are locally band preserving.
\item In $\ell^\infty$, a function is locally band preserving if and only if it is defined coordinatewise.
\end{enumerate}
\end{example}

A locally band preserving function is precisely one whose behaviour on a band is completely independent of anything outside the band. Given a locally band preserving function $f$ and a band $B$, one can define $f_B \colon \bP (\dom(f)) \to B$ by $f_B(x) = \bP(f(y))$ where $y \in \dom(f)$ with $\bP(y) = x$. Then $f(x) = f_B(\bP(x)) + f_{B^d}(\bP^d(x))$. This is the key idea that informs the following properties of locally band preserving functions. These properties are from \cite[Proposition 4.12.]{paper1}.

\begin{proposition}
Let $E$ be a vector lattice with the projection property. Let $f \colon [a, b] \to E$ be locally band preserving with $[c, d] \subseteq [a, b]$. The following statements hold.
\begin{enumerate}[(i)]
\item If $f$ is order bounded above on $[c, d]$ by $M$, then $\bP(f(x)) \leq \bP(M)$ holds for all $x \in [a, b]$ and all band projections $\bP$ such that $\bP(c) \leq \bP(x) \leq \bP(d)$.
\item If $f$ is order bounded below on $[c, d]$ by $m$, then $\bP(f(x)) \geq \bP(m)$ holds for all $x \in [a, b]$ and all band projections $\bP$ such that $\bP(c) \leq \bP(x) \leq \bP(d)$.
\item Let $x \in [a, b]$ and $\eps, \delta \in E^+$ be such that for all $y \in [a, b]$, $\abs{x - y} \leq \delta$ implies $\abs{f(x) - f(y)} \leq \eps$, then for any band projection $\bP$, $\bP(\abs{x-y}) \leq \bP(\delta)$ implies $\bP(\abs{f(x) -f(y)}) \leq \bP(\eps)$.
\end{enumerate}
\end{proposition}

Finally, we mention that the Intermediate Value Theorem and the Extreme Value Theorem were generalised for locally band preserving functions in \cite[Theorems 5.3 and 5.5]{paper1}.

\begin{theorem}[Intermediate Value Theorem] \label{t:IVT}
Let $E$ be a Dedekind complete vector lattice. If $f \colon [a, b] \to E$ is order continuous and locally band preserving, and $y \in [f(a) \wedge f(b), f(a) \vee f(b)]$, then there exists a $c \in [a, b]$ such that $f(c) = y$.
\end{theorem}

\begin{theorem}[Extreme Value Theorem] \label{t:EVT}
Let $E$ be a Dedekind complete vector lattice. If $f \colon [a, b] \to E$ is order continuous and locally band preserving, then there exist $c, d \in [a, b]$ such that for all $x \in [a, b]$, we have $f(c) \leq f(x) \leq f(d)$.
\end{theorem}

\subsection{Order differentiability and super order differentiability}

There are two natural ways of defining derivatives in the vector lattice setting that turn out to be related, but not equivalent. These derivatives were first studied in \cite{differentiation}, and then further investigated in \cite{paper1}. Between those two papers, there is a minor difference in whether $r \gg 0$ means $r$ is a positive invertible element or a weak order unit---we use the latter, and note that the results in \cite{differentiation} are still true. To obtain uniqueness of order derivatives with this change, it was necessary to assume $E$ satisfies the projection property in \cite[Proposition 3.7]{paper1}.

\begin{definition} \label{d:order_differentiable}
Let $E$ be a $\Phi$-algebra with the projection property. A function $f \colon \dom(f) \to E$ is called \emph{order differentiable at $c \in \dom(f)$} with derivative $f'(c)$ if there exists an $r \gg 0$ such that $(c - r, c + r) \subseteq \dom(f)$, and for every $\Delta \downarrow 0$ there exists an $\Eps \searrow 0$ with the property that for all $\eps \in \Eps$ there exists a $\delta \in \Delta$ satisfying
\[ x \in (c - r, c+ r) \text{ and } \abs{x - c} \leq \delta \implies \abs{f(x) - f(c) - (x - c) f'(c)} \leq \abs{x - c} \eps. \]
\end{definition}

\begin{definition} \label{d:super_order_differentiabile}
Let $E$ be a $\Phi$-algebra with the projection property. A function $f \colon \dom(f) \to E$ is called \emph{super order differentiable at $c\in \dom(f)$} with derivative $f'(c)$ if there exists an $r \gg 0$ such that $(c - r, c + r) \subseteq \dom(f)$, and for every $\Delta \downarrow 0$ there exists an $\Eps \searrow 0$ with the property that for all $\eps \in \Eps$ there exists a $\delta \in \Delta$ satisfying
\[ x \in \dom(f) \text{ and } \abs{x - c} \leq \delta \implies \abs{f(x) - f(c) - (x - c) f'(c)} \leq \abs{x - c} \eps. \]
\end{definition}

\begin{remark}
The above definitions differ only in whether the implication is required to hold only for $x \in (c - r, c + r)$ or for all $x \in \dom(f)$. These are not equivalent as it is possible for a net $(x_\alpha)$ to converge to $c$ without ever entering $(c - r, c + r)$. It is immediate from the definitions that $f$ is order differentiable at $c$ if and only if there exists an $r \gg 0$ such that $(c - r, c + r) \subseteq \dom(f)$ and the restriction $f \vert_{(c - r, c + r)}$ is super order differentiable at $c$.
\end{remark}

Order derivatives are always unique, see \cite[Remark 3.3]{differentiation} and \cite[Proposition 3.7]{paper1}. We also have available the sum and product rules for order differentiability from \cite[Theorems 3.15 and 4.3]{differentiation}, where the reader may also find chain and quotient rules.

\begin{proposition}\label{p:diff}
Let $E$ be a $\Phi$-algebra with the projection property. Suppose $f \colon \dom(f) \to E$ and $g \colon \dom(g) \to E$ are (super) order differentiable at $c \in \dom(f) \cap \dom(g)$. Then
\begin{enumerate}[(i)]
\item $f+g$ is (super) order differentiable at $c$ with $$(f+g)'(c) = f'(c) + g'(c),$$ and
\item $fg$ is (super) order differentiable at $c$ with $$(fg)'(c) = f'(c)g(c) + f(c)g'(c).$$
\end{enumerate}
\end{proposition}

A function being super order differentiable on an interval automatically makes it locally band preserving, see \cite[Theorem 4.9 and Corollary 4.10]{paper1}. Additionally, it was shown in \cite[Proposition 4.11]{paper1} that if $f \colon (a, b) \to E$ is order differentiable and locally band preserving, then $f'$ is necessarily locally band preserving too.

We now state a version of the Mean Value Theorem, which is \cite[Theorem 5.8]{paper1} along with the natural corollary of when it can be used to say that a zero derivative implies a function is constant. The assumption that the function is locally band preserving cannot be dropped (cf. \cite[Example 5.11]{paper1}).

\begin{theorem}[Mean Value Theorem] \label{t:MVT}
Let $E$ be a Dedekind complete $\Phi$-algebra, and suppose $a \ll b$. Let $f \colon [a, b] \to E$ be order continuous on $[a, b]$, order differentiable on $(a, b)$, and locally band preserving. For any $x, y \in [a, b]$, there exists $c \in [x \wedge y, x \vee y]$ such that
\[ (y - x) f'(c) = f(y) - f(x). \]
\end{theorem}

\begin{corollary}
Let $E$ be a Dedekind complete $\Phi$-algebra. Let $f \colon (a, b) \to E$ be locally band preserving and order differentiable with $f' = 0$. Then $f$ is constant.
\end{corollary}

\section{Construction of the Riemann integral}

\begin{notation}
Throughout this section and the remainder of the paper, $E$ always denotes a Dedekind complete $f$-algebra.
\end{notation}

We define a Riemann integral for order bounded functions $f \colon [a, b] \to E$, where $[a, b]$ is an order interval in $E$. At first, we use the Darboux sum construction of the integral, but it is shown in \Cref{p:riemann_sum_characterization} that the Riemann sum construction is equivalent.

\begin{definition}
A \emph{partition} $P$ of $[a, b]$ is a finite, totally ordered subset of $[a, b]$ containing $a$ and $b$, i.e.,
\[ P = \{ a = x_0 \leq x_1 \leq \cdots \leq x_n = b \}. \]

For any subinterval $I$ of $[a, b]$, we define
\[ m_I := \inf_{x \in I} f(x) \quad \text{and} \quad M_I := \sup_{x \in I} f(x). \]
If we need to specify the function under consideration, we will write $m_I(f)$ and $M_I(f)$, and when considering a partition $P$ as above, we write $m_i := m_{[x_{i-1}, x_i]}$ and $M_i := M_{[x_{i-1}, x_i]}$. 

Let $f \colon [a, b] \to E$ be a bounded function. Then the \emph{lower sum} $L(f, P)$ and the \emph{upper sum} $U(f, P)$ of $f$ with respect to $P$ are given by
\[ L(f, P) := \sum_{i=1}^n m_i(x_i - x_{i-1}) \quad \text{and} \quad U(f, P) := \sum_{i=1}^n M_i(x_i - x_{i-1}). \]

The \emph{lower integral} $L(f)$ and the \emph{upper integral} $U(f)$ of $f$ are defined as
\begin{align*}
L(f) & := \sup \{L(f, P) \colon P \text{ is a partition of } [a, b] \} \text{ and} \\
U(f) & := \inf \{U(f, P) \colon P \text{ is a partition of } [a, b] \}.
\end{align*}
If we wish to specify the interval, we write $L(f, [a, b])$ and $U(f, [a, b])$.

We say that $f$ is \emph{(Darboux) integrable} if $U(f) = L(f)$. In this case, we define the \emph{(Darboux) integral} of $f$ over $[a, b]$ to be
\[ \int_a^b f(x) dx := U(f) = L(f). \]
\end{definition}

We interest ourselves only in the integration of locally band preserving functions $f$. The decision to do so is motivated by the fact that a function can have a zero order derivative and yet not be constant---an example is provided in \cite[Example 5.11]{paper1}---unless the function is locally band preserving. Thus, in order to have a Fundamental Theorem of Calculus, assuming $f$ is locally band preserving is necessary. The following example exhibits some other problems that occur without this assumption.

\begin{example} \label{ex: swap-function}
Let $f \colon [(0, 0), (1, 1)] \to \bR^2$ be given by $f(x, y) := (y, x)$.
Let $P := \{ (0, 0), (1, 0), (1, 1) \}$ and $Q := \{ (0, 0), (0, 1), (1, 1) \}$. Then, $P \cup Q = \{ (0, 0), (1, 0), (0, 1), (1, 1) \}$ is not a partition of $[a, b]$ since $(1, 0)$ and $(0, 1)$ are incomparable. Since $L(f, P) = (0, 1)$ and $U(f, Q) = (1, 0)$, we see that $L(f, P) \not\leq U(f, Q)$. This also means $f$ is not integrable.
\end{example}

Note that in this example, $f$ is a uniformly continuous function and yet not integrable. Because $E$ is not totally ordered, $P \cup Q$ need not be a partition of $[a, b]$, and then one cannot utilise the classical argument:
\[ L(f, P) \leq L(f, P \cup Q) \leq U(f, P \cup Q) \leq U(f, Q). \]

This problem can be addressed by taking the total orderisation of $P \cup Q$ as a ``common refinement'' of $P$ and $Q$. However, this approach is only sensible when $f$ is locally band preserving, because the following two lemmas are needed to find a relationship between $L(f, P)$ and $L(f, Q)$, and $L(f, \totord(P \cup Q))$---and between the respective upper sums.

\begin{lemma} \label{l:lbp_extrema_preserved}
Let $f \colon [a \wedge c, b \vee d] \to E$ be order bounded and locally band preserving with $a \leq b$ and $c \leq d$, and let $\bP$ be a band projection.
\begin{enumerate}[(i)]
\item If $[\bP(a), \bP(b)] \subseteq [\bP(c), \bP(d)]$, then
\[ \bP(m_{[c, d]}) \leq \bP(m_{[a, b]}) \quad \text{and} \quad \bP(M_{[a, b]}) \leq \bP(M_{[c, d]}). \]
\item If $\bP(a) = \bP(c)$ and $\bP(b) = \bP(d)$, then
\[ \bP(m_{[a, b]}) = \bP(m_{[c, d]}) \quad \text{and} \quad \bP(M_{[a, b]}) = \bP(M_{[c, d]}). \]
\end{enumerate}
\end{lemma}
\begin{proof}
To prove (i), suppose that $[\bP(a), \bP(b)] \subseteq [\bP(c), \bP(d)]$. Then we have that
\begin{align*}
\bP(m_{[a, b]})
& = \bP \left( \inf \{ f(x): a \leq x \leq b \} \right) \\
& = \inf \{ \bP(f(x)): a \leq x \leq b \} \\
& = \inf \{ \bP(f(\bP(x) + \bP^d(c))): a \leq x \leq b \} \\
& = \inf \{ \bP(f(y + \bP^d(c))): \bP(a) \leq y \leq \bP(b)\} \\
& \geq \inf \{ \bP(f(y + \bP^d(c))): \bP(c) \leq y \leq \bP(d)\} \\
& = \bP(m_{[c, d]}),
\end{align*}
where the term $\bP^d(c)$ is added in the third step in order to guarantee that $\bP(x) + \bP^d(c) \in [a\wedge c, b \vee d]$. The final equality follows by reversing the previous steps with $a$ and $b$ interchanged with $c$ and $d$, respectively. The proof for the suprema is similar, and (ii) follows immediately by applying (i) twice.
\end{proof}

\begin{lemma} \label{l:darboux_sum_subset_inequality}
Let $f \colon [a, b] \to E$ be order bounded and locally band preserving. If $P$ and $Q$ are partitions of $[a, b]$ with $\bP(P) \subseteq \bP(Q)$ for some band projection $\bP$, then
\[ \bP(L(f, P)) \leq \bP(L(f, Q)) \quad \text{and} \quad \bP(U(f, Q)) \leq \bP(U(f, P)). \]
\end{lemma}
\begin{proof}
Suppose $P=\{x_0, \dots, x_n\}$. We prove the lower sum inequality as the proof of the upper sum inequality is similar. First suppose $P \subseteq Q$. Without loss of generality, we may assume that $Q$ has one element more than $P$. Say $Q = P \cup \{y\}$ with $x_{k-1} \leq y \leq x_k$ for some $y \in E$ and $1 \leq k \leq n$. Then,
\begin{align*}
m_{[x_{k-1}, x_k]}(x_k - x_{k-1})
& = m_{[x_{k-1}, x_k]}(y - x_{k-1}) + m_{[x_{k-1}, x_k]}(x_k - y) \\
& \leq m_{[x_{k-1}, y]}(y - x_{k-1}) + m_{[y, x_k]}(x_k - y),
\end{align*}
so that $L(f, P) \leq L(f, Q)$.

Now suppose that $\bP(P) \subseteq \bP(Q)$ for some band projection $\bP$. There exists a subpartition $R \defeq \{y_0, \dots, y_n\}$ of $Q$ such that $\bP(x_i) = \bP(y_i)$ for $1 \leq i \leq n$. It follows from the first part of this proof that $L(f, R) \leq L(f, Q)$. Therefore, by using \Cref{l:lbp_extrema_preserved}(ii),
\begin{align*}
\bP(L(f, P))
& = \bP \left( \sum_{i = 1}^n m_{[x_{i-1}, x_i]} (x_i - x_{i-1}) \right) = \bP \left( \sum_{i = 1}^n m_{[y_{i-1}, y_i]} (y_i - y_{i-1}) \right) \\
& = \bP(L(f, R)) \leq \bP(L(f, Q)).\qedhere
\end{align*}
\end{proof}

We now dedicate ourselves to turning the collection of all partitions of $[a, b]$ into a directed set such that $P, Q \preceq \totord(P \cup Q)$, which will allow us to see the integral as an order limit of Darboux sums.

\begin{definition}
Let $\cP$ denote the collection of all partitions of $[a, b]$. We define a preorder on $\cP$ by $P \preceq Q$  if and only if there exists a band decomposition $E = \bigoplus_{i = 1}^m B_i$ such that $\bP_i(P) \subseteq \bP_i(Q)$ for $1 \leq i \leq m$.
\end{definition}

\begin{remark}
Classically, partitions are often directed using the notion of mesh. This results in the same integral, but we prefer using $\preceq$ as it illustrates the difference between the classical case (where $\preceq$ reduces to $\subseteq$) and ours.
\end{remark}

\begin{lemma} \label{l:totord_union_is_refinement}
Let $f \colon [a, b] \to E$ be order bounded and locally band preserving. If $P$ and $Q$ are partitions of $[a, b]$, then $P, Q \preceq \totord(P \cup Q)$.
\end{lemma}
\begin{proof}
This follows immediately by using \Cref{l:totord_decomposition}.
\end{proof}

\begin{lemma} \label{l:darboux_sum_nets_monotone}
Let $f \colon [a, b] \to E$ be order bounded and locally band preserving, and let $P$ and $Q$ be partitions of $[a, b]$. If $P \preceq Q$, then
\[ L(f, P) \leq L(f, Q) \quad \text{and} \quad U(f, Q) \leq U(f, P). \]
\end{lemma}
\begin{proof}
By definition of $P \preceq Q$, there exists a band decomposition $E = \bigoplus_{i = 1}^m B_i$ such that $\bP_i(P) \subseteq \bP_i(Q)$ for $1 \leq i \leq m$. Using \Cref{l:darboux_sum_subset_inequality},
\[ L(f, P) = \sum_{i = 1}^m \bP_i(L(f, P)) \leq \sum_{i = 1}^m \bP_i(L(f, Q)) = L(f, Q). \]
The other inequality can be shown in the same way.
\end{proof}

\begin{proposition} \label{p:partitions_directed_darboux_sums_monotone}
$(\cP, \preceq)$ is an upwards directed set. Moreover, if $f \colon [a, b] \to E$ is order bounded and locally band preserving, then we have $L(f, P) \uparrow L(f)$ and $U(f, P) \downarrow U(f)$.
\end{proposition}
\begin{proof}
\Cref{l:totord_union_is_refinement} shows that $(\cP, \preceq)$ is an upwards directed set, and the rest of the claim is \Cref{l:darboux_sum_nets_monotone}.
\end{proof}

As a consequence of \Cref{p:partitions_directed_darboux_sums_monotone}, $f$ is integrable if and only if $$U(f, P) - L(f, P) \downarrow 0.$$ This immediately yields the following integrability condition, mimicking the classical integrability condition.

\begin{proposition} \label{p:integrability_condition}
Let $f \colon [a, b] \to E$ be order bounded and locally band preserving. Then $f$ is integrable if and only if there exists an $\Eps \searrow 0$ such that for every $\eps \in \Eps$ there exists a partition $P$ of $[a, b]$ such that $U(f, P) - L(f, P) \leq \eps$.
\end{proposition}

To end this section, we demonstrate that the integral could have equivalently been defined as a limit of Riemann sums.

\begin{definition}
A \emph{tagged partition} $(P, \{c_i\})$ of $[a, b]$ consists of a partition $P = \{x_0, \dots, x_n\}$ and a set of points $\{c_i: 1 \leq i \leq n\}$ such that $c_i \in [x_{i-1}, x_i]$ for $1 \leq i \leq n$. For $f \colon [a, b] \to E$, define the \emph{Riemann sum} of $f$ with respect to the tagged partition $(P, \{c_i\})$ to be
\[ R(f, P, \{c_i\}) := \sum_{i = 1}^n f(c_i)(x_i - x_{i-1}). \]

Let $\cT$ be the collection of all tagged partitions of $[a, b]$. We write $$(P, \{c_i\}) \preceq (Q, \{d_i\}) \quad\mbox{if and only if}\quad P \preceq Q,$$ and this also defines a preorder on $\cT$ that directs $\cT$ upwards. We say $f$ is \emph{(Riemann) integrable} with integral $I \in E$ if $R(f, P, \{c_i\}) \to I$.
\end{definition}

\begin{proposition} \label{p:riemann_sum_characterization}
Let $f \colon [a, b] \to E$ be order bounded and locally band preserving. Then $f$ is Darboux integrable if and only if $f$ is Riemann integrable. In this case, the integrals coincide.
\end{proposition}
\begin{proof}
Suppose $U(f, P) \to I$ and $L(f, P) \to I$. Then there exists an $\Eps \searrow 0$ such that for every $\eps \in \Eps$ there exists a partition $P_\eps$ such that
\[ P \succeq P_\eps \implies U(f, P) - I \leq \eps \text{ and } I - L(f, P) \leq \eps. \]
Let $\eps \in \Eps$. Choose any $d_i$ that makes $(P_\eps, \{d_i\})$ a tagged partition of $[a, b]$. Then, for any $(P, \{c_i\}) \succeq (P_\eps, \{d_i\})$, we have that $P \succeq P_\eps$ and so
\[ R(f, P, \{c_i\}) - I \leq U(f, P) - I \leq \eps \text{ and } I - R(f, P, \{c_i\}) \leq I - L(f, P) \leq \eps. \]
Therefore, $\abs*{R(f, P, \{c_i\}) - I} \leq \eps$ and $R(f, P, \{c_i\}) \to I$.

Now suppose that $R(f, P, \{c_i\}) \to I$. Then there exists an $\Eps \searrow 0$ such that for every $\eps \in \Eps$, there exists a $(P_\eps, \{c_i\})$ such that
\[ (P, \{d_i\}) \succeq (P_\eps, \{c_i\}) \implies \abs{R(f, P, \{d_i\}) - I} \leq \eps. \]
Let $\eps \in \Eps$ with $P_\eps \defeq \{x_0, \dots, x_n\}$. For $1 \leq i \leq n$, let $d_i \in [x_{i-1}, x_i]$ then
\[ I - \eps \leq R(f, P_\eps, \{d_i\}) \leq I + \eps. \]
Taking the supremum and infimum over all choices of $\{d_i\}$ respectively yields
\[ I - \eps \leq U(f, P_\eps) \leq I + \eps \text{ and } I - \eps \leq L(f, P_\eps) \leq I + \eps. \]
Therefore, $U(f, P) \downarrow I$ and $L(f, P) \uparrow I$.
\end{proof}

\section{Properties of the Riemann integral} \label{s:properties_of_the_integral}

In what follows, we obtain generalisations of the basic classical properties of the Riemann integral.

\begin{proposition}
Let $f \colon [a \wedge c, b \vee d] \to E$ be order bounded and locally band preserving with $a \leq b$, $c \leq d$, $\bP(a) = \bP(c)$, and $\bP(b) = \bP(d)$ for some band projection $\bP$. If $f$ is integrable both on $[a, b]$ and on $[c, d]$, then
\[ \bP \left( \int_a^b f(x) dx \right) = \bP \left( \int_c^d f(x) dx \right). \]
\end{proposition}
\begin{proof}
For any partition $P$ of $[a, b]$, there exists a partition $Q$ of $[c, d]$ such that $\bP(P) = \bP(Q)$, since $\bP(a) = \bP(c)$ and $\bP(b) = \bP(d)$. By \Cref{l:lbp_extrema_preserved}(ii), 
\[ \bP (L(f, P)) = \bP(L(f, Q)) \leq \bP \left( \int_c^d f(x) dx \right). \]
Since $\bP$ is order continuous, taking the limit over all partitions $P$ of $[a, b]$ yields
\[ \bP \left( \int_a^b f(x) dx \right) \leq \bP \left( \int_c^d f(x) dx \right). \]
The same argument can be used to show the reverse inequality.
\end{proof}

\begin{proposition}
Let $f \colon [a, b] \to E$ be order bounded and locally band preserving with $c \in [a, b]$. Then $f$ is integrable on $[a, b]$ if and only if $f$ is integrable on $[a, c]$ and $[c, b]$. In this case,
\[ \int_a^b f(x) dx = \int_a^c f(x) dx + \int_c^b f(x) dx. \]
\end{proposition}
\begin{proof}
Suppose that $f$ is integrable on $[a, b]$. Then there exists an $\Eps \searrow 0$ such that for every $\eps \in \Eps$ there exists a partition $P$ of $[a, b]$ such that $U(f, P) - L(f, P) \leq \eps$. Fix $\eps \in \Eps$ and let $P \defeq \{x_0, \dots, x_n\}$ be the corresponding partition. Now, $Q \defeq P \wedge c = \{y_0 \defeq x_0 \wedge c, \dots, y_n \defeq x_n \wedge c \}$ is a partition of $[a, c]$. Fix $1 \leq i \leq n$. Then
\[ \bP_{c \leq x_{i-1}}(y_i - y_{i-1}) = \bP_{c \leq x_{i-1}}(c - c) = 0, \] and
\[ \bP_{c > x_{i-1}}(x_{i-1}) = \bP_{c > x_{i-1}} (y_{i-1}) \leq \bP_{c > x_{i-1}}(y_i) \leq \bP_{c > x_{i-1}}(x_i). \]
By \Cref{l:lbp_extrema_preserved}(i),
\begin{align*}
(M_{[y_{i-1}, y_i]} - m_{[y_{i-1}, y_i]}) (y_i - &y_{i-1})
 = \bP_{c > x_{i-1}} \left( (M_{[y_{i-1}, y_i]} - m_{[y_{i-1}, y_i]}) (y_i - y_{i-1}) \right) \\
& \leq \bP_{c > x_{i-1}} \left( (M_{[x_{i-1}, x_i]} - m_{[x_{i-1}, x_i]}) (x_i - x_{i-1}) \right) \\
& \leq (M_{[x_{i-1}, x_i]} - m_{[x_{i-1}, x_i]}) (x_i - x_{i-1}).
\end{align*}
Therefore, $U(f, Q) - L(f, Q) \leq U(f, P) - L(f, P) \leq \eps$. This proves that $f$ is integrable on $[a, c]$. Similarly, $f$ is integrable on $[c, b]$.

Conversely, suppose $f$ is integrable on $[a, c]$ and on $[c, b]$. Then, by \Cref{p:integrability_condition}, there exist $\Eps_a, \Eps_b \searrow 0$ such that for any $\eps_a \in \Eps_a$ and $\eps_b \in \Eps_b$ there exist partitions $P_1$ and $P_2$ of $[a, c]$ and $[c, b]$, respectively, such that
\[ U(f, P_1) - L(f, P_1) \leq \eps_a \quad \text{and} \quad U(f, P_2) - L(f, P_2) \leq \eps_b. \]
Then, $P \defeq P_1 \cup P_2$ is a partition of $[a, b]$ such that
\[ U(f, P) - L(f, P) = [U(f, P_1) - L(f, P_1)] + [U(f, P_2) - L(f, P_2)] \leq \eps_a + \eps_b. \]
Since $\Eps_a + \Eps_b \searrow 0$, $f$ is integrable on $[a, b]$. Moreover,
\begin{align*}
\int_a^b f(x) dx
& \leq U(f, P) \leq L(f, P) + \eps_a + \eps_b \\
& = L(f, P_1) + L(f, P_2) + \eps_a + \eps_b \\
& \leq \int_a^c f(x) dx + \int_c^b f(x) dx + \eps_a + \eps_b,
\end{align*}
showing that $\int_a^b f(x) dx \leq \int_a^c f(x) dx + \int_c^b f(x) dx$. The reverse inequality can be shown in a similar fashion.
\end{proof}

\begin{proposition} \label{p:integral_properties}
Suppose $f, g \colon [a, b] \to E$ are order bounded and locally band preserving. If both are Riemann integrable on $[a, b]$ and $c \in E$, then the following holds.
\begin{enumerate}[(i)]
\item $f + g$ is Riemann integrable on $[a, b]$ and
\[ \int_a^b (f + g)(x) dx = \int_a^b f(x) dx + \int_a^b g(x) dx. \]
\item $cf$ is Riemann integrable on $[a, b]$ and
\[ \int_a^b (cf)(x) dx = c \int_a^b f(x) dx. \]
\item If $f \leq g$, then $\int_a^b f(x) dx \leq \int_a^b g(x) dx$.
\item $\abs{f} \colon x \mapsto \abs{f(x)}$ is integrable and $\abs*{\int_a^b f(x) dx} \leq \int_a^b \abs{f(x)} dx$.
\end{enumerate}
\end{proposition}
\begin{proof}
To prove (i), note that for any subinterval $I$ of $[a, b]$, $M_I(f + g) \leq M_I(f) + M_I(g)$ and $m_I(f) + m_I(g) \leq m_I(f + g)$. Therefore, for any partition $P$ of $[a, b]$, we have that
\[ L(f, P) + L(g, P) \leq L(f + g, P) \leq U(f + g, P) \leq U(f, P) + L(f, P). \]
Taking the limit over $P$ throughout yields
\[ \int_a^b f(x) dx + \int_a^b g(x) dx \leq L(f + g) \leq U(f + g) \leq \int_a^b f(x) dx + \int_a^b g(x) dx. \]
Thus, $L(f+g) = U(f+g)$ and
\[
\int_a^b (f+g)(x)dx = \int_a^b f(x) dx + \int_a^b g(x) dx.
\]

The proof of (ii) is similar except we use that for any subinterval $I$ of $[a, b]$ we have that $m_I(cf) = c^+ m_I(f) - c^- M_I(f)$ and $M_I(cf) = c^+ M_I(f) - c^- m_I(f)$. Thus, for any partition $P$ of $[a, b]$, we have that
\[ c^+ L(f, P) - c^- U(f, P) = L(cf, P) \leq U(cf, P) = c^+ U(f, P) - c^- L(f, P). \]
Taking the limit over $P$ throughout yields
\[ c \int_a^b f(x) dx = L(cf) \leq U(cf) = c \int_a^b f(x) dx. \]
Thus, $L(cf) = U(cf)$ and 
\[
\int_a^b cf(x) dx = c \int_a^b f(x) dx.
\]

Taking the limit over $P$ in $U(f, P) \leq U(g, P)$ immediately yields (iii), and the second half of (iv) follows from (iii) after we prove the first half. By \Cref{p:integrability_condition}, there exists an $\Eps \searrow 0$ such that for all $\eps \in \Eps$ there exists a partition $P$ of $[a, b]$ satisfying $U(f, P) - L(f, P) \leq \eps$. For any subinterval $I$ of $[a, b]$, we have that for all $x, y \in I$,
\[ \abs{\abs{f(x)} - \abs{f(y)}} \leq \abs{f(x) - f(y)}, \]
from which it follows that $M_I(\abs{f}) - m_I(\abs{f}) \leq M_I(f) -m_I(f)$. Therefore,
\[ U(\abs{f}, P) - L(\abs{f}, P) \leq U(f, P) - L(f, P) \leq \eps.\qedhere \]
\end{proof}

\begin{proposition}
Let $f, g \colon [a, b] \to E$ be order bounded and locally band preserving. If $f$ and $g$ are both integrable, then $fg$ is integrable.
\end{proposition}
\begin{proof}
Let $K \in E$ be such that $\abs{f(x)}, \abs{g(x)} \leq K$ for all $x \in [a, b]$. 

(i) We first assume $f = g \geq 0$. Since $f$ is integrable, there exists an $\Eps \searrow 0$ such that for every $\eps \in \Eps$ there exists a partition $P$ of $[a, b]$ such that $U(f, P) - L(f, P) \leq \eps$.

Let $\eps \in \Eps$ and $P := \{x_0, \dots, x_n\}$ as above. Then, for any subinterval $I$ of $[a, b]$, we have that $m_I(f^2) = m_I(f)^2$ and $M_I(f^2) = M_I(f)^2$. Therefore,
\begin{align*}
U(&f^2, P) - L(f^2, P)
 = \sum_{i = 1}^n (M_{[x_{i-1}, x_i]}(f^2) - m_{[x_{i-1}, x_i]}(f^2))(x_i - x_{i-1}) \\
& = \sum_{i = 1}^n (M_{[x_{i-1}, x_i]}(f)^2 - m_{[x_{i-1}, x_i]}(f)^2)(x_i - x_{i-1}) \\
& = \sum_{i = 1}^n (M_{[x_{i-1}, x_i]}(f) + m_{[x_{i-1}, x_i]}(f))(M_{[x_{i-1}, x_i]}(f) - m_{[x_{i-1}, x_i]}(f))(x_i - x_{i-1}) \\
& \leq 2 K \sum_{i = 1} (M_{[x_{i-1}, x_i]}(f) - m_{[x_{i-1}, x_i]}(f))(x_i - x_{i-1}) \\
& = 2 K (U(f, P) - L(f, P)) \\
& \leq 2 K \eps.
\end{align*}
Since $2 K \Eps \searrow 0$, this shows $f^2$ is integrable.

(ii) Now assume $f = g$. Since $f$ is integrable, $\abs{f}$ is integrable by \Cref{p:integral_properties}, and therefore $fg = f^2 = \abs{f}^2$ is integrable by part (i).

(iii) Without any additional assumptions on $f$ and $g$, showing that $fg$ is integrable follows from
\[ fg = \tfrac{1}{4} ((f+g)^2 - (f-g)^2).\qedhere \]
\end{proof}

Next we will define what it means for a net of functions to converge uniformly, and show that a uniform limit of integrable functions is again integrable. Note that to sensibly define what it means to converge uniformly, we must have the same regulating set $\Eps$ when considering convergence at each point, which need not be true when one has pointwise convergence.

\begin{definition} \label{d:uniform_convergence}
Let $f_\alpha, f \colon \dom(f) \to E$. We say that $f_\alpha \to f$ \emph{uniformly} if there exists an $\Eps \searrow 0$ such that for all $\eps \in \Eps$, there exists an $\alpha_0$ satisfying
\[ x \in \dom(f) \text{ and } \alpha \geq \alpha_0 \implies \abs{f_\alpha(x) - f(x)} \leq \eps. \]
\end{definition}

\begin{proposition}
Let $f_\alpha, f \colon [a, b] \to E$ be order bounded and locally band preserving functions with $f_\alpha \to f$ uniformly. If $f_\alpha$ is integrable for each $\alpha$, then $f$ is integrable and
\[ \int_a^b f_\alpha(x) dx \to \int_a^b f(x) dx. \]
\end{proposition}
\begin{proof}
Since $f_\alpha \to f$ uniformly, there exists an $\Eps \searrow 0$ such that for every $\eps \in \Eps$ there exists an $\alpha_\eps$ satisfying
\[ y \in [a, b] \text{ and } \alpha \geq \alpha_\eps \implies \abs{f_\alpha(y) - f(y)} \leq \eps. \]
Let $\alpha \geq \alpha_\eps$. Then, for any subinterval $I \subseteq [a, b]$,
\[ m_I(f_\alpha) - \eps \leq m_I(f) \leq M_I(f) \leq M_I(f_\alpha) + \eps, \]
and hence, for any partition $P = \{x_0, \dots, x_n\}$ of $[a, b]$, we have that
\begin{align*}
U(f, P)
& = \sum_{i = 1}^n M_{[x_{i-1}, x_i]}(f) (x_i - x_{i-1}) \\
& \leq \sum_{i = 1}^n (M_I(f_\alpha) + \eps)(x_i - x_{i- 1}) \\
& = U(f_\alpha, P) + \eps (b - a),
\end{align*}
and, similarly, $L(f, P) \geq L(f_\alpha, P) - \eps(b - a)$. Taking the limit over all partitions $P$ yields
\[ \int_a^b f_\alpha(x) dx - \eps(b - a) \leq L(f) \leq U(f) \leq \int_a^b f_\alpha(x) + \eps(b - a). \]
Finally, taking the limit over all $\alpha$ shows that
\[ L(f) = U(f) = \lim_\alpha \int_a^b f_\alpha(x) dx. \qedhere\]
\end{proof}

The following two propositions provide examples of integrable functions.

\begin{proposition} \label{p:monotone_integrable}
Let $f \colon [a, b] \to E$ be order bounded, monotone, and locally band preserving. Then $f$ is integrable.
\end{proposition}
\begin{proof}
Assume, without loss of generality, that $f$ is increasing. Define $$\Eps \defeq \{ \tfrac{1}{n} (b - a)(f(b) - f(a)): n \in \bN \}.$$ Let $\eps := \frac{1}{n}(b - a)(f(b) - f(a))$, and define $P = \{x_0, \dots, x_n\}$, where $x_i \defeq a + \frac{i}{n} (b - a)$. Then,
\begin{align*}
U(f, P) - L(f, P)
& = \sum_{i = 1}^n (M_{[x_{i-1}, x_i]} - m_{[x_{i-1}, x_i]})(x_i - x_{i-1}) \\
& = \sum_{i = 1}^n (f(x_i) - f(x_{i-1})) ({\textstyle\frac{1}{n}} (b - a)) \\
& = {\textstyle\frac{1}{n}} (b - a) (f(b) - f(a)) \\
& = \eps.
\end{align*}
Since $\Eps \searrow 0$, $f$ is integrable.
\end{proof}

\begin{proposition} \label{p:uniformly_continuous_integrable}
Let $f \colon [a, b] \to E$ be uniformly order continuous and locally band preserving. Then $f$ is integrable.
\end{proposition}
\begin{proof}
Since $f$ is uniformly order continuous, there exists an $\Eps \searrow 0$ such that for every $\eps \in \Eps$ there exists an $N_\eps \in \bN$ such that
\[ x, y \in [a, b] \text{ and } \abs{x - y} \leq \frac{1}{N_\eps} (b - a) \implies \abs{f(x) - f(y)} \leq \eps. \]
Let $\eps \in \Eps$, and define $P \defeq \{x_0, \dots, x_{N_\eps} \}$, where $x_k \defeq a + \frac{k}{N_\eps} (b - a)$ for $0 \leq k \leq N_\eps$. Then,
\begin{align*}
U(f, P) - L(f, P)
& = \sum_{k = 1}^{N_\eps} \left( M_{[x_{k-1}, x_k]} - m_{[x_{k-1}, x_k]} \right) (x_k - x_{k - 1}) \\
& \leq \sum_{k = 1}^{N_\eps} \eps (x_k - x_{k-1}) \\
& = \eps (b - a).
\end{align*}
Since $(b - a) \Eps \searrow 0$, $f$ is integrable.
\end{proof}

\begin{remark} \label{r:open_question}
Note that it is an open question whether the assumption of uniform order continuity in \Cref{p:uniformly_continuous_integrable} can be weakened to order continuity. In \cite[Example 4.1]{paper1}, an example is given of an order continuous function $f \colon [a, b] \to E$ that is unbounded (and hence not uniformly order continuous), but we have no such examples for locally band preserving functions. Thus, there are three possibilities for an order continuous and locally band preserving function $f \colon [a, b] \to E$: (i) $f$ must be uniformly order continuous (and hence integrable), (ii) $f$ must be integrable but need not be uniformly order continuous, or (iii) $f$ need not be integrable.
\end{remark}

\section{Integrals with incomparable endpoints}

We now turn our attention to defining integrals where the endpoints $a$ and $b$ are not necessarily comparable. Classically, we have that if $a \not\leq b$, then $b < a$ and we define $\int_a^b f(x) dx := - \int_b^a f(x)$. In our setting, we instead have the decomposition
\[ E = B_{a < b} \oplus B_{b < a} \oplus B_{a = b}. \]

Intuitively, we can integrate normally where $a < b$, we must integrate from $b$ to $a$ and add a minus sign where $b < a$, and the integral must be zero where $a = b$. This informs the following definition.

\begin{definition}
Let $f \colon [a \wedge b, a \vee b] \to E$ be order bounded and locally band preserving, where $a, b \in E$. We say that $f$ is \emph{Riemann integrable from $a$ to $b$} if $f$ is integrable on $[a \wedge b, a \vee b]$. In this case, we define
\[ \int_a^b f(x) dx \defeq \bP_{a < b} \left( \int_{a \wedge b}^{a \vee b} f(x) dx \right) - \bP_{b < a} \left( \int_{a \wedge b}^{a \vee b} f(x) dx \right). \]
\end{definition}

The next proposition is an immediate consequence of the defintion.

\begin{proposition}
Let $f \colon [a \wedge b, a \vee b] \to E$ be order bounded, locally band preserving, and integrable from $a$ to $b$. Then
\[ \abs*{\int_a^b f(x) dx} = \abs*{\int_{a \wedge b}^{a \vee b} f(x) dx}. \]
\end{proposition}

Next we provide generalisations of the properties from \Cref{s:properties_of_the_integral}, where we now allow the endpoints of the integral to be incomparable. The proofs have been omitted as the generalisations follow immediately.

\begin{proposition}
Let $f \colon [a \wedge c, b \vee d] \to E$ be order bounded and locally band preserving with $\bP(a) = \bP(c)$ and $\bP(b) = \bP(d)$. If $f$ is integrable both from $a$ to $b$ and from $c$ to $d$, then
\[ \bP \left( \int_a^b f(x) dx \right) = \bP \left( \int_c^d f(x) dx \right). \]
\end{proposition}

\begin{proposition}
Let $f \colon [a \wedge b, a \vee b] \to E$ be order bounded and locally band preserving with $c \in [a \wedge b, a \vee b]$. Then $f$ is integrable from $a$ to $b$ if and only if $f$ is integrable from $a$ to $c$ and from $c$ to $b$. In this case,
\[ \int_a^b f(x) dx = \int_a^c f(x) dx + \int_c^b f(x) dx. \]
\end{proposition}

\begin{proposition}
Let $f, g \colon [a \wedge b, a \vee b] \to E$ be order bounded and locally band preserving. If both $f$ and $g$ are integrable from $a$ to $b$ and $c \in E$, then the following holds.
\begin{enumerate}[(i)]
\item $f + g$ is integrable from $a$ to $b$ and
\[ \int_a^b (f+g)(x) dx = \int_a^b f(x) dx + \int_a^b g(x) dx. \]

\item $cf$ is integrable on from $a$ to $b$ and
\[ \int_a^b (cf)(x) dx = c \int_a^b f(x) dx. \]

\item $\abs{f}$ is integrable from $a$ to $b$.

\item $fg$ is integrable from $a$ to $b$.
\end{enumerate}
\end{proposition}

\section{The Fundamental Theorem of Calculus}

In this section, we first provide a Mean Value Theorem for Integrals in \Cref{t:MVT_for_integrals}, which will be used to prove the Fundamental Theorem of Calculus---of which parts 1 and 2 are \Cref{t:FTOC,t:FTOC2}, respectively. As this section connects our theory of integration to the theory of differentiation from \cite{differentiation}, we now assume $E$ to be a Dedekind complete $\Phi$-algebra.

\begin{theorem} \label{t:MVT_for_integrals}
Let $f \colon [a, b] \to E$ be locally band preserving and order continuous. If $f$ is integrable, then, for any $x, y \in [a, b]$, there exists $c \in [x \wedge y, x \vee y]$ such that
\[ (y - x) f(c) = \int_x^y f(z) dz. \]
\end{theorem}
\begin{proof}
Assume $x \leq y$. By \Cref{t:EVT}, there exist $d_1, d_2 \in [x, y]$ such that $f(d_1) \leq f(z) \leq f(d_2)$ for all $z \in [x, y]$. Since $$(y - x) f(d_1) \leq \int_x^y f(z) dz \leq (y - x) f(d_2),$$ applying \Cref{t:IVT} to the function $w \mapsto (y - x) f(w)$ shows that there exists $c \in [x, y]$ such that $(y - x) f(c) = \int_x^y f(z) dz$.

Let $x, y \in [a, b]$ without restriction. Then there exists $c \in [x \wedge y, x \vee y]$ such that
\[ (x \vee y - x \wedge y) f(c) = \int_{x \wedge y}^{x \vee y} f(z) dz. \]
It follows that
\begin{align*}
\int_x^y f(z) dz
& = \bP_{x < y} \left( \int_{x \wedge y}^{x \vee y} f(z) dz \right) - \bP_{y < x} \left( \int_{x \wedge y}^{x \vee y} f(z) dz \right) \\
& = \bP_{x < y} ((x \vee y - x \wedge y) f(c)) - \bP_{y < x} ((x \vee y - x \wedge y) f(c)) \\
& = \bP_{x < y} ((y - x) f(c)) - \bP_{y < x} ((x - y) f(c)) \\
& = (y - x) f(c). \qedhere
\end{align*}
\end{proof}

\begin{theorem} \label{t:FTOC}
Let $a \ll b$. Suppose $f \colon [a, b] \to E$ is order continuous, locally band preserving, and integrable. Then $F(x) \defeq \int_a^x f(z) dz$ is uniformly continuous on $[a, b]$ and super order differentiable on $(a, b)$ with $F' = f$.
\end{theorem}
\begin{proof}
Fix $x \in (a, b)$. Let $\Delta \downarrow 0$. Since $f$ is order continuous at $x$, there exists an $\Eps \searrow 0$ such that for every $\eps \in \Eps$ there exists a $\delta \in \Delta$ satisfying
\[ y \in [a, b] \text{ and } \abs{y - x} \leq \delta \implies \abs{f(y) - f(x)} \leq \eps. \]
Let $y \in [a, b]$ be such that $\abs{y - x} \leq \delta$. By \Cref{t:MVT_for_integrals}, there exists a $c \in [x \wedge y, x \vee y]$ such that
\[ F(y) - F(x) = \int_x^y f(z) dz = (y - x) f(c). \]
Since $\abs{c - x} \leq \delta$, it follows that
\begin{align*}
\abs{F(y) - F(x) - (y - x)f(x)}
& = \abs{y - x} \abs{f(c) - f(x)} \\
& \leq \abs{y - x} \eps.
\end{align*}
This shows that $F$ is super order differentiable on $(a, b)$ with $F' = f$.

To see that $F$ is uniformly order continuous on $[a, b]$, note that $f$ is order bounded by \Cref{t:EVT}, say $\abs{f(x)} \leq K$ for all $x \in [a, b]$. Define $\Eps \defeq K \Delta$. Let $\eps = K \delta \in \Eps$, then if $x, y \in [a, b]$ with $\abs{x - y} \leq \delta$,

\begin{align*}
\abs*{F(x) - F(y)}
& = \abs*{\int_y^x f(z) dz} = \abs*{\int_{x \wedge y}^{x \vee y} f(z) dz}  \leq K \abs*{x \vee y - x \wedge y} \\
& = K \abs{x - y}  \leq K \delta = \eps.
\end{align*}
This proves $F$ is uniformly continuous on $[a, b]$.
\end{proof}

\begin{theorem} \label{t:FTOC2}
Suppose $a \ll b$. Let $f \colon [a, b] \to E$ be order bounded and locally band preserving, and let $F \colon [a, b] \to E$ be order continuous on $[a, b]$, order differentiable on $(a, b)$ with $F' = f$, and locally band preserving. If $f$ is integrable, then, for any $x, y \in [a, b]$,
\[ \int_x^y f(z) dz = F(y) - F(x). \]
\end{theorem}
\begin{proof}
Assume $x \leq y$. Consider any partition $P \defeq \{x_0, \dots, x_n\}$ of $[x, y]$. For $1 \leq i \leq n$, by \Cref{t:MVT}, there exists a $y_i \in [x_{i-1}, x_i]$ such that
\[ F(x_i) - F(x_{i-1}) = f(y_i)(x_i - x_{i-1}),\]
and hence
\[ m_i (x_i - x_{i-1}) \leq F(x_i) - F(x_{i-1}) \leq M_i (x_i - x_{i-1}).\]
Summing over $i$ then yields $L(f, P) \leq F(y) - F(x) \leq U(f, P)$. Taking the limit over $P$ gives us that $\int_x^y f(z) dz = F(y) - F(x)$.

Let $x, y \in [a, b]$ without restriction. Then, using that $F$ is locally band preserving,
\begin{align*}
\int_x^y f(z) dz
& = \bP_{x < y} \left( \int_{x \wedge y}^{x \vee y} f(z) dz \right) - \bP_{y < x} \left( \int_{x \wedge y}^{x \vee y} f(z) dz \right) \\
& = \bP_{x < y} (F(x \vee y) - F(x \wedge y)) - \bP_{y < x} (F(x \vee y) - F(x \wedge y)) \\
& = \bP_{x < y} (F(y) - F(x)) - \bP_{y < x} (F(x) - F(y)) \\
& = F(y) - F(x). \qedhere
\end{align*}
\end{proof}

\begin{remark}
\cite[Example 5.11]{paper1} shows that the assumption that $F$ is locally band preserving in \Cref{t:FTOC2} is necessary. Although the antiderivative $F$ defined in \Cref{t:FTOC} will always be locally band preserving due to it being super order differentiable, other antiderivatives need not be and for these the result will fail.
\end{remark}

Now that we have a Fundamental Theorem of Calculus, we can generalise classic integration techniques such as $u$-substitution and integration by parts. These are given in the following two propositions.

\begin{proposition} \label{p:u-substitution}
Let $a \ll b$ and $c \ll d$. Let $g \colon [a, b] \to E$ and $f \colon [c, d] \to E$ be locally band preserving, order continuous, and integrable. Suppose $G \colon [a, b] \to [c, d]$ is locally band preserving, order continuous on $[a, b]$, and order differentiable on $(a, b)$ with $G' = g$. If $(f \circ G) g$ is Riemann integrable on $[a, b]$, then
\[ \int_a^b f(G(x)) g(x) dx = \int_{G(a)}^{G(b)} f(u) du. \]
\end{proposition}
\begin{proof}
Let $F(x) := \int_c^x f(u) du$. By \Cref{t:FTOC}, $F$ is uniformly continuous on $[c, d]$ and super order differentiable on $(c, d)$ with $F' = f$. Note that $G$ differs from a super order differentiable function by a constant using \Cref{t:FTOC,t:FTOC2}, and hence $G$ is also super order differentiable on $(a, b)$. By the chain rule \cite[Theorem 4.3(iii)]{differentiation}, $F \circ G$ is super order differentiable on $(a, b)$ with $(F \circ G)' = (f \circ G) g$. By assumption, $(f\circ G)g$ is integrable and using \Cref{t:FTOC2} now yields
\[ \int_a^b f(G(x)) g(x) dx = F(G(b)) - F(G(a)) = \int_{G(a)}^{G(b)} f(u) du. \]
\end{proof}

\begin{remark}
In \Cref{p:u-substitution}, it is clear that $(f \circ G) g$ is order continuous, but in light of \Cref{r:open_question} it is not known whether this guarantees $(f \circ G) g$ is integrable and we must therefore assume it.
\end{remark}

\begin{proposition}
Let $a \ll b$ and $c \ll d$. Let $g \colon [a, b] \to E$ and $f \colon [c, d] \to E$ be locally band preserving and uniformly order continuous functions. Suppose $G \colon [a, b] \to [c, d]$ is uniformly continuous and super order differentiable on $(a, b)$ with $G' = g$. Then $(f \circ G) g$ is Riemann integrable on $[a, b]$ and
\[ \int_a^b f(G(x)) g(x) dx = \int_{G(a)}^{G(b)} f(u) du. \]
\end{proposition}
\begin{proof}
Since $f$, $G$ and $g$ are all uniformly continuous, $(f \circ G) g$ is uniformly continuous on $[a, b]$. By \Cref{t:FTOC}, $f$ has an antiderivative $F \colon [c, d] \to E$ which is uniformly continuous on $[c, d]$ and super order differentiable on $(c, d)$ with $F' = f$. Then, $F \circ G$ is uniformly continuous on $[a, b]$ and super order differentiable on $(a, b)$ with $(F \circ G)' = (f \circ G) g$ by the chain rule \cite[Theorem~4.3(iii)]{differentiation}. By \Cref{t:FTOC2},
\[ \int_a^b f(G(x)) g(x) \,dx= F(G(b)) - F(G(a)) = \int_{G(a)}^{G(b)} f(u) du. \qedhere \]
\end{proof}

\begin{proposition}
Suppose $a \ll b$. Let $f, g, f', g' \colon [a, b] \to E$ be locally band preserving, order continuous on $[a, b]$, and integrable. If $f'$ and $g'$ are the order derivatives of $f$ and $g$ on $(a, b)$, respectively. Then
\[ \int_a^b f(x) g'(x) dx = f(b) g(b) - f(a) g(a) - \int_a^b f'(x) g(x) dx. \]
\end{proposition}
\begin{proof}
We know that $fg$ is order differentiable on $(a, b)$ with $(fg)' = f g' + f' g$ by the product rule \cite[Theorem~4.3(ii)]{differentiation}, and $(fg)'$ is integrable since it is a product and sum of integrable functions. By \Cref{t:FTOC2},
\[ \int_a^b f(x) g'(x) + f'(x) g(x) dx = f(b) g(b) - f(a) g(a).\qedhere  \]
\end{proof}

\noindent
{\bf Acknowledgements.} The authors would like to express their gratitude to the Erasmus$+$ ICM Grant for facilitating the productive visits between Leiden University and the University of Pretoria. 

\bibliographystyle{alpha}
\bibliography{refs.bib}

\end{document}